\documentclass[12pt, a4paper, leqno]{amsart}
\usepackage{enumerate,enumitem}
\usepackage[all]{xy}
\usepackage{mathtools}
\usepackage{amsmath,amssymb,mathrsfs,geometry,color,bm, enumerate}
\usepackage{pb-diagram}
\usepackage{bbm}

\usepackage{geometry}
\usepackage[colorlinks=true,linkcolor=blue,citecolor=blue]{hyperref}
\newcommand{\doublespace}
   {\addtolength{\baselineskip}{0.25\baselineskip}}

\theoremstyle{definition}
\newtheorem{theorem}{Theorem}[section]

\newtheorem{lemma}[theorem]{Lemma}
\newtheorem{definition}[theorem]{Definition}
\newtheorem{assumption}[theorem]{Assumption}
\newtheorem{remark}[theorem]{Remark}

\newtheorem{example}[theorem]{Example}

\newtheorem{corollary}[theorem]{Corollary}
\numberwithin{equation}{section}

\newcommand{\R}{\mathbb{R}}   
   
\newcommand{\N}{\mathbb{N}}

\newcommand{\calF}{{\mathcal{F}}}

\newcommand{\calM}{{\mathcal{M}}}

\title[Free extreme value distributions]{A note on convergence of densities to free extreme value distributions}

\author{Yamato Kindaichi and Yuki Ueda}


\address{Yamato Kindaichi: Sapporo, Hokkaido, Japan}
\email{e.ifswb2mussw2@gmail.com}

\address{
Yuki Ueda: 
Department of Mathematics, Hokkaido University of Education, 9 Hokumon-cho, Asahikawa, Hokkaido 070-8621, Japan}
\email{ueda.yuki@a.hokkyodai.ac.jp}

\begin{document}

\maketitle  

\doublespace
\pagestyle{myheadings} 

\begin{abstract}
The concept of free extreme value distributions as universal limit laws for the spectral maximum of free noncommutative real random variables was discovered by Ben Arous and Voiculescu in 2006. This paper contributes to study the convergence of densities towards free extreme value distributions under the von Mises condition for sample distributions.
\end{abstract}



\section{Introduction}

Extreme value theory plays a pivotal role in analyzing maximum or minimum within a large dataset of samples. Its most notable feature is its universality, as established by the Fisher-Tippett-Gnedenko theorem \cite{FT28, G43} (see also \cite{R}). This theorem asserts that the limit laws of the maximum for a large number of samples are classified into only three types: Fr\'{e}chet, Weibull and Gumbel distributions. The study of convergence towards extreme value distributions is one of significant topics in pure mathematics of extreme value theory. Smith \cite{S82} and Omey \cite{O88} explored the rate of uniform convergence in the Fr\'{e}chet case. Subsequently, Bartholm\'{e} and Swan \cite{BS13} also examined similar outcomes using Stein's method. Hall \cite{Hall79} has already studied the rate of uniform convergence between extreme values in normal i.i.d. random variables and the Gumbel distributions. More comprehensively, Kusumoto and Takeuchi \cite{KT20} employed Stein's method to analyze the rate of uniform convergence in the Gumbel case. Generally, convergence in distributions does not necessarily imply the convergence of their corresponding densities (even if densities exist). Haan and Resnick \cite{HR82} proved that the normalize maximum converges to the densities of extreme value distributions under the so-called von Mises conditions. Additionally, Omey \cite{O88} investigated the rate of convergence for the density of normalized maximum towards extreme values under certain regularly varying conditions for a sample distribution.

In the framework of Voiculescu's free probability theory (see \cite{NS06,V86} for details), Ben Arous and Voiculescu \cite{BV06} developed a theory to analyze the maximum of a large number of non-commutative real random variables on some $W^*$-probability space (i.e. selfadjoint operators affiliated with some von Neumann algebra). Let $(\calM,\tau)$ be a $W^*$-probability space. According to \cite{BV06}, for any free non-commutative real random variables $X,Y$ on $(\calM,\tau)$, the spectral distribution function $F_{X\lor Y}(x)= \tau(\mathbf{1}_{(-\infty,x]}(X\lor Y))$ of the maximum $X\lor Y$ (with respect to spectral order \cite{A89,O71}) is given by
$$
F_{X\lor Y}(x)=\max \{F_X(x)+F_Y(x)-1,0\} =: (F_X \Box\hspace{-.93em}\lor F_Y)(x),
$$
where $F_X(x)=\tau(\mathbf{1}_{(-\infty,x]}(X))$ and $F_Y$ is similarly defined. In general, for any distribution functions $F$ and $G$ on $\R$, we define
$$
F \Box\hspace{-.93em}\lor G :=  \max\{F+G-1,0\} \quad \text{and} \quad F^{\Box\hspace{-.55em}\lor  n} := \underbrace{F\Box\hspace{-.93em}\lor  \cdots \Box\hspace{-.93em}\lor  F}_{n\text{ times}}.
$$ The operation $\Box\hspace{-.77em}\lor$ is called the {\it free max-convolution}. Similarly as classical extreme value theory, the limiting distributions of the normalized maximum $W_n=  ((X_1\lor \cdots \lor X_n) -b_n)/a_n$ (for some $a_n>0$ and $b_n\in \R$) of a large number of freely independent, identically distributed non-commutative real random variables $X_1,\dots, X_n$ on some $W^*$-probability space, are also characterized by just three types:
\begin{align*}
\Phi_\alpha^{\rm free}(x) := 
\begin{cases}
 (1-x^{-\alpha})\mathbf{1}_{[1,\infty)}(x), & \alpha>0 \quad \text{(free Fr\'{e}chet type});\\
 \{1-(-x)^{-\alpha}\}\mathbf{1}_{[-1,0]}(x) + \mathbf{1}_{(0,\infty)}(x), &  \alpha < 0 \quad \text{(free Weibull type)};\\
 (1-e^{-x})\mathbf{1}_{[0,\infty)}(x) & \alpha=0  \quad \text{(free Gumbel type)}.
 \end{cases}
\end{align*}
The distributions $\Phi_\alpha^{\rm free}$ are called the {\it free extreme value distributions}. Rigorously, for a distribution function $F$ on $\R$, there exist $a_n>0$, $b_n\in \R$ and a non-degenerate distribution function $G$ such that
$$
F^{\Box\hspace{-.55em}\lor n} (a_n x +b_n) \to G(x), \quad \text{as} \quad n\to\infty, \quad x\in \R,
$$
then $G$ is the free extreme value distribution. 

Since \cite{BV06}, free extreme value theory has continued to advance from various perspectives. Ben Arous and Kargin \cite{BK10} clarified the relationships between the convergence of free point processes and the max-domains of attraction of free extreme value distributions. Benaych-Georges and Cabanal-Duvillard \cite{BGCD10} established a certain random matrix model such that the empirical eigenvalue distributions of random matrices converge to the free extreme value distribution as the matrix size tends to infinity. Grela and Nowak \cite{GN20} provided explicit calculations for several extreme random matrix ensembles. Also they presented a formula, showing the equivalence of free extreme distributions to the POT (Peak-Over-Threshold) method in classical probability. 
Recently, the author \cite{U22-2} has discovered a close relationship between limit theorems for classical and free max-convolutions. This results are the max-analogue presented in Bercovici and Pata's seminal work \cite{BP99}. In \cite{U21} and \cite{HU21} (with Hasebe), the author further found interesting relations between additive convolution semigroups and max-convolution semigroups in classical, free and boolean cases via Tucci, Haagerup and M\"{o}ller limit theorem \cite{T10, HM13}.

In accordance with the author's work \cite{U22}, the uniform convergence of $F^{\Box\hspace{-.55em}\lor n} (a_n \cdot +b_n)$ to the free extreme value distribution $G$ was proven under certain analytic assumptions regarding the density of the normalized maximum $W_n$, and the rate of convergence has been provided via Stein's method. We prompt a natural question here. Can we assert the convergence of densities of $W_n$ towards the density of free extreme value distribution? To be more precise, if the density function $w_n$ of the normalized maximum $W_n$ exists for almost all $n\in \N$, our inquiry pertains to whether $w_n(x)$ converges to the density of free extreme value distributions as characterized by 
\begin{align*}
\varphi_\alpha^{\rm free}(x) := 
\begin{cases}
\alpha x^{-\alpha-1} \mathbf{1}_{(1,\infty)}(x),  & \alpha>0 \quad \text{(free Fr\'{e}chet density)}; \\
-\alpha (-x)^{-\alpha-1} \mathbf{1}_{(-1,0)}(x), & \alpha<0 \quad \text{(free Weibull density)};\\
 e^{-x} \mathbf{1}_{(0,\infty)}(x), &  \alpha=0  \quad \text{(free Gumbel density)}.
\end{cases}
\end{align*}
In this paper, we establish the convergence of densities towards the free extreme values distributions, under the von Mises condition and a few analytic assumptions regarding sample distributions. 
\begin{theorem}
Consider a distribution function $F=\exp\{-e^{-\phi}\}$ where $\phi$ is a function with some analytic assumptions. Let us define $h_\alpha(x)$ by
$$
h_\alpha(x):= \begin{cases}
x\phi'(x)-\alpha, & \alpha>0, \ \omega_F=\infty;\\
(\omega_F-x) \phi'(x) + \alpha, & \alpha<0,\ \omega_F<\infty;\\
(1/\phi'(x))', &\alpha=0,
\end{cases}
$$
where $\omega_F:=\sup\{x\in \R: F(x) < 1\}$ and we understand $\omega_F=\infty$ when $F(x)<1$ for all $x\in \R$. Assume that there exists a non-increasing function $g$ such that $g(x)\to 0$ as $x\to \omega_F $ and $|h_\alpha|\le g$ ({\it von Mises condition}). Then we obtain the followings.

\begin{itemize}
\item {\bf The case $\alpha > 0$ and $\omega_F=\infty$}. If $a_n>0$ satisfies $F(a_n)=e^{-\frac{1}{n}}$, then
$$
\sup_{x>1} |w_n(x)- \varphi_\alpha^{\rm free}(x)| \le O(n^{-1} \lor g(a_n)).
$$ 
\item {\bf The case $\alpha<0$ and $\omega_F<\infty$}. Let us consider $\alpha<-1$. If $a_n>0$ satisfies $F(\omega_F-a_n)=e^{-\frac{1}{n}}$ and $b_n=\omega_F$, then
$$
\sup_{-1<x<0} |w_n(x)- \varphi_\alpha^{\rm free}(x)| \le O(n^{-1} \lor g(\omega_F-a_n)).
$$
When $\alpha<0$ in general, it does not hold (see Remark \ref{rem:uniformconvergesfails}). 
\item {\bf The case $\alpha=0$}. If $b_n\in \R$ satisfies $F(b_n)=e^{-\frac{1}{n}}$ and $a_n=\frac{F(b_n)}{nF'(b_n)}$, then
$$
\sup_{x>0} |w_n(x)- \varphi_0^{\rm free}(x)| \le O(n^{-1} \lor g(b_n)).
$$
\end{itemize}
\end{theorem}
For details, see Theorem \ref{thm:LUC:Frechet} for $\alpha>0$, Theorem \ref{thm:Weibull}, Corollary \ref{cor:W} and Remark \ref{rem:uniformconvergesfails} for $\alpha<0$ and Theorem \ref{thm:Gumbel} for $\alpha=0$.


\section{The case $\alpha >0$  (Fr\'{e}chet type)}

In this section, we consider a distribution function $F$ defined as $F=\exp\left\{-e^{-\phi}\right\}$, where $\omega_F=\infty$ and $\phi$ is a $C^1$-function with $\phi' > 0$ in a neighborhood of $\infty$. Choose $a_n = F^{\leftarrow} (e^{-\frac{1}{n}})$, where $F^{\leftarrow} (y) := \inf \{x\in \R: F(x) \ge y\}$. Since $F$ is continuous, we get $F(a_n)=e^{-\frac{1}{n}}$, equivalently, $\phi(a_n)=\log n$. Since $\phi$ is differentiable, the function $x\mapsto F^{\Box\hspace{-.55em}\lor n}(a_nx)$ has the density function $w_n$ given by
\begin{align*}
w_n(x) =na_n F'(a_nx) =na_n \phi'(a_nx)(-\log F(a_nx)) F(a_nx),
\end{align*}
for all $x> A_n:=a_n^{-1} F^{\leftarrow}(1-n^{-1}) \in (-\infty,1)$. Since $a_n\sim F^{\leftarrow} (1-n^{-1})$, we get $A_n\to 1$ as $n\to \infty$.

We further assume the von Mises condition for sample distributions $F$.

\begin{assumption}\label{asumpt:Frechet}
For $\alpha>0$, we define 
\begin{align*}
h_\alpha(x):=x\phi'(x)-\alpha=\frac{xF'(x)}{F(x)(-\log F(x))}-\alpha, \qquad x>0. 
\end{align*}
Assume $h_\alpha(x)\rightarrow0$ as $x \to \infty$ (see \cite[(2.42)]{R}).
\end{assumption}

Under Assumption \ref{asumpt:Frechet}, we get $F^n(a_n\cdot )\xrightarrow{w} \Phi_\alpha$ by \cite[Pages 107--108]{R}, where 
$$\Phi_\alpha(x)=\exp(-x^{-\alpha})\mathbf{1}_{(0,\infty)}(x), \quad \alpha>0,$$ is the Fr\'{e}chet distribution. Due to \cite[Theorem 6.12]{BV06}, one can also see that $F^{\Box\hspace{-.55em}\lor n}(a_n\cdot )\xrightarrow{w} \Phi_\alpha^{\rm free}$. We summarize the above discussion as follows.

\begin{lemma}\label{lem:vM}
Under Assumption \ref{asumpt:Frechet}, we have $F^{\Box\hspace{-.55em}\lor n}(a_n\cdot )\xrightarrow{w} \Phi_\alpha^{\rm free}$.
\end{lemma}

According to the above discussion, we are interested in the following class of sample distributions.

\begin{definition}
Let us consider $\alpha>0$. Denote by $\calF_\alpha$ the set of all distribution functions $F$ such that
\begin{enumerate}
\item[(F-1)] $F=\exp\left\{-e^{-\phi}\right\}$, where $\omega_F=\infty$ and $\phi$ is a $C^1$-function with $\phi' > 0$ in a neighborhood of $\infty$;
\item[(F-2)] there exists a nonincreasing continuous function $g$ on $(0,\infty)$ such that $g(x) \rightarrow 0$ as $x\rightarrow\infty$ and $|h_\alpha(x)|\le g(x)$ for all $x>0$.
\end{enumerate}
\end{definition}
Obviously, Assumption \ref{asumpt:Frechet} achieves from the condition (F-2). Therefore if $F\in\calF_\alpha$ then $F^{\Box\hspace{-.55em}\lor n}(a_n\cdot )\xrightarrow{w} \Phi_\alpha^{\rm free}$ by Lemma \ref{lem:vM}.

For any $F\in \calF_\alpha$, we prove the uniform convergence of the density $w_n$ of $x\mapsto F^{\Box\hspace{-.55em}\lor n}(a_nx)$ to the free Fr\'{e}chet density $\varphi_\alpha^{\rm free}$ on $(1,\infty)$ that is a domain in which $\varphi_\alpha^{\rm free}>0$. To show this, we prepare the following inequality for the free Fr\'{e}chet distributions.

\begin{lemma}\label{lem:Unif_Frechet}
For any $\alpha_1,\alpha_2>0$, we have 
$$\sup_{x\in \R} |\Phi_{\alpha_1}^{\rm free}(x)-\Phi_{\alpha_2}^{\rm free}(x)| \le  e^{-1} \ \dfrac{|\alpha_2-\alpha_1|}{ \alpha_1 \lor \alpha_2}.$$
\end{lemma}
\begin{proof}
Without loss of generality, we may assume that $\alpha_1 <\alpha_2$. Note that $|\Phi_{\alpha_1}^{\rm free}(x)-\Phi_{\alpha_2}^{\rm free}(x)|=0$ for all $x\le 1$. Since $\frac{\partial}{\partial\beta} x^{-\beta} = - x^{-\beta} \log x$ for $\beta>0$ and $x>1$, we get
\begin{align*}
\sup_{x>1} |\Phi_{\alpha_1}^{\rm free}(x)-\Phi_{\alpha_2}^{\rm free}(x)| 
&= \sup_{x > 1} |x^{-\alpha_1}- x^{-\alpha_2}|\\
&\le \sup_{x>1} \sup_{\beta \in [\alpha_1, \alpha_2]} |-x^{-\beta}\log x| |\alpha_2-\alpha_1|\\
&=(\alpha_2-\alpha_1) \sup_{x>1} x^{-\alpha_2} \log x=e^{-1}\ \frac{\alpha_2-\alpha_1}{\alpha_2}.
\end{align*}
\end{proof}

\begin{theorem}\label{thm:LUC:Frechet}
Let $F\in \calF_\alpha$, $a_n$ and $g$ be defined as above. Then $w_n$ converges uniformly to the free Fr\'{e}chet density $\varphi_\alpha^{\rm free}$ on $(1,\infty)$. More precisely, we get
\begin{align*}
\sup_{x>1} |w_n(x)- \varphi_\alpha^{\rm free} (x)| \le O(n^{-1} \lor g(a_n)),
\end{align*}
for sufficiently large $n$. 
\end{theorem}
\begin{proof}
For all $x>1$ and for sufficiently large $n\in \N$ (such that infinitely many $a_n$ are in a neighborhood of $\infty$), we have
\begin{align*}
|w_n(x)-\varphi_\alpha^{\rm free}(x)|
&\le |na_n\phi'(a_nx)(-\log F(a_nx))-\alpha x^{-\alpha-1}| + \alpha x^{-\alpha-1}(1-F(a_nx))\\
&\le \underbrace{|na_n\phi'(a_nx)(-\log F(a_nx))-\alpha x^{-\alpha-1}|}_{=: I} + \alpha (1-e^{-\frac{1}{n}}).
\end{align*}
On the other hands, we obtain
\begin{align*}
I &\le \left|\frac{n(-\log F(a_nx))}{x} (a_nx\phi'(a_nx)-\alpha)  \right|  + \left| \frac{(n(-\log F(a_nx))}{x} \alpha-\alpha x^{-\alpha-1}\right|\\
&\le g(a_nx) \frac{n\exp(-\phi(a_nx))}{x} + |n(-\log F(a_nx))-x^{-\alpha}| \alpha x^{-1} \qquad \text{(by condition (F-2))}\\
&\le g(a_n) + \alpha \underbrace{|n(-\log F(a_nx))-x^{-\alpha}|}_{=:J},
\end{align*}
where the last inequality holds since $\phi'>0$ in a neighborhood of $\infty$ and 
\[
\frac{n\exp(-\phi(a_nx))}{x} \le  n \exp (-\phi(a_n)) = 1.
\]

According to \cite[Page 108]{R}, for $x\ge 1$, we have
\[
(\alpha-g(a_n))\log x \le \phi(a_nx) - \phi (a_n) \le (\alpha+g(a_n))\log x, 
\]
which implies that $\Phi_{\alpha-g(a_n)}^{\rm free}(x) \le 1 +n \log F(a_nx) \le \Phi_{\alpha+g(a_n)}^{\rm free}(x)$. Thus,
\begin{align}\label{eq:nlogF}
\Phi_\alpha^{\rm free}(x)- \Phi_{\alpha+g(a_n)}^{\rm free}(x)\le n(-\log F(a_nx)) -x^{-\alpha} \le \Phi_\alpha^{\rm free}(x)-\Phi_{\alpha-g(a_n)}^{\rm free}(x)
\end{align}
for all $x\ge 1$, and therefore 
$$
J \le e^{-1} \frac{2g(a_n)}{\alpha+g(a_n)} \le\frac{2}{\alpha e} g(a_n)
$$
by Lemma \ref{lem:Unif_Frechet}. Hence, the desired result holds true.
\end{proof}

\begin{example}\label{ex:Frechet}
As below, since all $F$ are differentiable, we take $a_n>0$ such that $F(a_n)=e^{-\frac{1}{n}}$. Let $w_n$ be the density of the function $x\mapsto F^{\Box\hspace{-.55em}\lor n}(a_nx)$.
\begin{enumerate} 
\item ({\it Fr\'{e}chet distribution}) Consider $F=\Phi_\alpha$ for $\alpha>0$. The choice of $a_n$ leads to $a_n= n^{\frac{1}{\alpha}}$ for $n\in \N$. As per the definition of $h_\alpha$, we have $h_\alpha(x)=0$. Hence,
\begin{align*}
\sup_{x>1} |w_n(x)-\varphi_\alpha^{\rm free}(x)| \le O(n^{-1}) 
\end{align*}
by Theorem \ref{thm:LUC:Frechet}. Note that the uniform convergence of $w_n$ to $\varphi_\alpha^{\rm free}$ on $(-\infty,1]$ is not true. Actually, since $A_n=(-n\log(1-n^{-1}))^{-\frac{1}{\alpha}}$, we get 
$$
w_n(A_n+)=\alpha \left\{-n \log\left(1-n^{-1} \right)\right\}^{1+\frac{1}{\alpha}} \left(1-n^{-1}\right)  \to \alpha
$$ 
as $n\to\infty$. Thus, we conclude
$$
\lim_{n\to \infty}\sup_{A_n<x\le 1}|w_n(x)-\varphi_\alpha^{\rm free}(x)|=\lim_{n\to \infty}\sup_{A_n<x\le 1}|w_n(x)|\neq 0.
$$

\item ({\it Log-logistic distribution})
Consider $\alpha>0$ and $F(x)=\{1-(1+x^\alpha)^{-1}\}\mathbf{1}_{(0,\infty)}(x)$. The choice of $a_n>0$ leads to $a_n=(e^{\frac{1}{n}}-1)^{-\frac{1}{\alpha}}$ for $n\in \N$.
By definition of $h_\alpha$, one can see
\[
|h_\alpha(x)|=\alpha-\frac{\alpha}{(1+x^\alpha) (-\log(1-(1+x^\alpha)^{-1}))} \le \frac{\alpha}{1+x^\alpha} =:g(x), \qquad x>0,
\]
and $g(a_n)=\alpha(1-e^{-\frac{1}{n}})$. According to Theorem \ref{thm:LUC:Frechet}, we also obtain
\begin{align*}
\sup_{x>1}|w_n(x)-\varphi_\alpha^{\rm free}(x)| \le O(n^{-1}) 
\end{align*}

\item ({\it Cauchy distribution})
Consider $F(x)=\frac{1}{2}+\frac{1}{\pi} \text{Tan}^{-1}x$ for $x\in \R$. The choice of $a_n$ leads to $a_n=\tan( \pi e^{-\frac{1}{n}}-\frac{\pi}{2})$ for $n\in\N$. One can see that $F^{\Box\hspace{-.55em}\lor n}(a_n\cdot) \xrightarrow{w} \Phi_1^{\rm free}$. By definition of $h_\alpha$, we observe
\begin{align*}
|h_1(x)| &= \left|\frac{x}{\pi (1+x^2)\left( \frac{1}{2}+\frac{1}{\pi} \text{Tan}^{-1}x\right)\left(-\log \left( \frac{1}{2}+\frac{1}{\pi} \text{Tan}^{-1}x\right) \right)} - 1 \right| \\
&\le 1 -\frac{x}{\pi (1+x^2)\left(-\log \left( \frac{1}{2}+\frac{1}{\pi} \text{Tan}^{-1}x\right) \right)}=:g(x)
\end{align*}
for any $x>0$. Furthermore, $g(x)\rightarrow0$ as $x\rightarrow\infty$. This establishes 
\begin{align*}
g(a_n) &=1+\frac{n}{2\pi}\sin (2\pi e^{-\frac{1}{n}}) \\
&= 1 + \frac{n}{2\pi} \sin \left( 2\pi -\frac{2\pi}{n}+ \frac{\pi}{n^2} + O(n^{-3})\right)\\
&= 1+ \frac{n}{2\pi} \sin \left(-\frac{2\pi}{n}+\frac{\pi}{n^2} + O(n^{-3})\right)\\
&=1+\frac{n}{2\pi} \left\{ \left(-\frac{2\pi}{n}+\frac{\pi}{n^2} + O(n^{-3})\right) +O(n^{-3})\right\} \\
&=\frac{1}{2n} +O(n^{-2}).
\end{align*}
Therefore we get $\sup_{x>1}|w_n(x)-\varphi_1^{\rm free}(x)| \le O(n^{-1})$ by Theorem \ref{thm:LUC:Frechet}.
\end{enumerate}
\end{example}


\section{The case $\alpha<0$ (Weibull type)}

In this section, we consider a distribution function $F$ defined as $F=\exp\left\{-e^{-\phi}\right\}$, where $\omega_F<\infty$ and $\phi$ is a $C^1$-function with $\phi' > 0$ in a left neighborhood of $\omega_F$. Define $a_n' = F^{\leftarrow} (e^{-\frac{1}{n}})$. Thus we have $F(a_n')=e^{-\frac{1}{n}}$, equivalently, $\phi(a_n')=\log n$. Let us set $a_n=\omega_F-a_n'$ and $b_n=\omega_F$ for all $n\in \N$. Since $\phi$ is differentiable, the density $w_n$ of the function $x\mapsto F^{\Box \hspace{-.55em}\lor n}(a_nx+b_n)$ exists and is given by
$$
w_n(x)= na_n \phi'(a_nx+b_n) (-\log F(a_nx+b_n)) F(a_nx+b_n),
$$
for all $A_n<x<0$, where
$$
A_n= \frac{F^{\leftarrow}(1-n^{-1})-b_n}{a_n} \quad \text{and} \quad A_n\to -1 \quad \text{as} \quad n\to\infty.
$$
Similarly as in Fr\'{e}chet case, we assume the von Mises condition in Weibull case.
\begin{assumption} \label{Assumpt:W}
For $\alpha<0$, we define
$$
h_\alpha(x):=(\omega_F-x) \phi'(x)+\alpha =\frac{(\omega_F-x) F'(x)}{F(x)(-\log F(x))}+\alpha.
$$
Assume $h_\alpha(x)\to 0$ as $x\to \omega_F-0$.
\end{assumption}

Under Assumption \ref{Assumpt:W} and \cite[Pages 59-60]{R}, we get $F^n(a_n\cdot +b_n)\xrightarrow{w} \Phi_\alpha$ as $n\to \infty$, where 
$$\Phi_\alpha(x):= \exp(-(-x)^{-\alpha})\mathbf{1}_{(-\infty,0)}(x)+ \mathbf{1}_{[0,\infty)}(x), \quad \alpha<0,$$ is the Weibull distribution. Due to \cite{BV06}, it is equivalent to $F^{\Box\hspace{-.55em}\lor n}(a_n\cdot +b_n)\xrightarrow{w} \Phi_\alpha^{\rm free}$ as $n\to \infty$.

\begin{lemma} \label{lem:W}
Under Assumption  \ref{Assumpt:W}, we have $F^{\Box\hspace{-.55em}\lor n}(a_n\cdot +b_n)\xrightarrow{w} \Phi_\alpha^{\rm free}$ as $n\to \infty$.
\end{lemma}

According to the above discussion, we set the collection of certain sample distributions.

\begin{definition}
Consider $\alpha<0$. Denote by $\calF_\alpha$ the set of all distribution functions $F$ such that
\begin{enumerate}
\item[(W-1)] $F=\exp\{-e^{-\phi}\}$, where $\phi$ is a $C^1$-function with $\phi'>0$ in a left neighborhood of $\omega_F$. 
\item[(W-2)] there exists a nonincreasing function $g$ such that $g(x)\to 0$ as $x\to \omega_F$ and $|h_\alpha(x)|\le g(x)$ for all $x < \omega_F$.
\end{enumerate}
\end{definition}

It follows from Lemma \ref{lem:W} that $F\in \calF_\alpha$ implies that $F^{\Box\hspace{-.55em}\lor n}(a_n\cdot +b_n) \xrightarrow{w} \Phi_\alpha^{\rm free}$. The same strategy as in the Fr\'{e}chet case can be used to prove density convergence in the Weibull case. 

\begin{theorem}\label{thm:Weibull}
Consider $\alpha<-1$. Let $F\in \calF_\alpha$, $a_n,b_n$ and $g$ be defined as above. Then we have
$$
\sup_{-1<x<0} |w_n(x)- \varphi_\alpha^{\rm free}(x)| \le O(n^{-1} \lor g(\omega_F-a_n)),
$$
for sufficiently large $n$. 
\end{theorem}

\begin{proof}
Consider $\beta=-\alpha>1$. Let us put $F_\ast(x):= F(\omega_F-x^{-1})$ and $\phi_\ast(x):=\phi(\omega_F-x^{-1})$. Then $F_\ast= \exp\{-e^{-\phi_\ast}\}$ and $\phi_\ast$ is a $C^1$-function with $\phi_\ast'>0$ in a neighborhood of $\infty$. Moreover, we define the function $h_{\beta}^\ast(x):=x\phi_\ast'(x)-\beta$. Then we get
$$
h_{\beta}^\ast(x) = h_\alpha(\omega_F-x^{-1}) \to 0, \quad \text{as} \quad x\to \infty,
$$
by Assumption \ref{Assumpt:W}. Thus $h_\beta^{\ast}$ satisfies the von Mises condition in Fr\'{e}chet case (see Assumption \ref{asumpt:Frechet}). And also, $|h_\beta^\ast(x)| \le g(\omega_F-x^{-1})$ by (W-2). If we put $a_n^\ast:=a_n^{-1}$, then $a_n^\ast = F_\ast^{\leftarrow}(e^{-\frac{1}{n}})$ (and hence $F_\ast (a_n^\ast)=e^{-\frac{1}{n}}$) and $F_\ast^{\Box\hspace{-.55em}\lor n}(a_n^\ast \cdot) \xrightarrow{w} \Phi_\beta^{\rm free}$. Let $w_n^\ast$ be the density of $x\mapsto F_\ast^{\Box\hspace{-.55em}\lor n}(a_n^\ast x)$. Since one can see that $w_n^\ast(x)=x^{-2}w_n(-x^{-1})$ and $\varphi_\beta^{\rm free}(x)=x^{-2} \varphi_\alpha^{\rm free}(-x^{-1})$, we have
\begin{equation}\label{eq:sup_Weibull}
\begin{split}
\sup_{-1<x<0} |w_n(x)-\varphi_\alpha^{\rm free}(x)| 
&= \sup_{x>1} |w_n(-x^{-1})- \varphi_\alpha^{\rm free}(-x^{-1})|\\
&=\sup_{x>1} x^2|w_n^\ast(x)- \varphi_\beta^{\rm free}(x)|.
\end{split}
\end{equation}

For any $x>1$, we get
\begin{align*}
x^2&|w_n^\ast(x)- \varphi_\beta^{\rm free}(x)|\\
&= x^2 |na_n^\ast \phi_\ast '(a_n^\ast x) (-\log F_\ast(a_n^\ast x))F_\ast(a_n^\ast x)-\beta x^{-\beta-1}|\\
&= x |(h_\beta^\ast(a_n^\ast x)+\beta) n(-\log F_\ast(a_n^\ast x))F_\ast(a_n^\ast x) - \beta x^{-\beta}|\\
&\le x g(\omega_F - (a_n^\ast x)^{-1}) n(-\log F_\ast(a_n^\ast x)) + \beta \underbrace{|nx(-\log F_\ast(a_n^\ast x))-x^{-\beta +1}|}_{=:I}\\
&\hspace{4mm} + \beta x^{-\beta+1}(1-F_\ast(a_n^\ast x))\\
&\le g(\omega_F-a_n) nx \exp(-\phi_\ast(a_n^\ast x)) +\beta I + \beta  (1-e^{-\frac{1}{n}}).
\end{align*}
For sufficiently large $n$ and for any $x>1$, we get $\{x\exp(-\phi_\ast(a_n^\ast x))\}'=\exp(-\phi_\ast(a_n^\ast x)) (1-\beta - h_\alpha(a_n^\ast x))<0$, and therefore
$$
nx \exp(-\phi_\ast(a_n^\ast x)) \le n \exp(-\phi_\ast(a_n^\ast)) =1.
$$
A proof similar to that of Lemma \ref{lem:Unif_Frechet} yields that for $\beta_1,\beta_2>1$,
\begin{align}\label{eq:xPhi}
\sup_{x>1} | x\Phi_{\beta_1}^{\rm free}(x)- x\Phi_{\beta_2}^{\rm free}(x)|  \le e^{-1} \frac{|\beta_2-\beta_1|}{(\beta_1-1)\lor (\beta_2-1)}.
\end{align}
Since the inequality \eqref{eq:nlogF} implies that for any $x>1$,
$$
x\Phi_\beta^{\rm free}(x)- x\Phi_{\beta+g(\omega_F-a_n)}^{\rm free} (x)\le nx(-\log F_\ast(a_n^\ast x))-x^{-\beta +1} \le x\Phi_\beta^{\rm free}(x)- x\Phi_{\beta-g(\omega_F-a_n)}^{\rm free}(x),
$$
we get $$I \le e^{-1} \dfrac{2g(\omega_F-a_n)}{\beta +g(\omega_F-a_n)-1} \le   \dfrac{2}{e(\beta -1)}g(\omega_F-a_n)$$ by \eqref{eq:xPhi}. Finally, the equation \eqref{eq:sup_Weibull} implies that
$$
\sup_{-1<x<0} |w_n(x)-\varphi_\alpha^{\rm free}(x)| \le \left(1 + \frac{2\beta}{e(\beta-1)}\right) g(\omega_F-a_n) +\beta (1-e^{-\frac{1}{n}}),
$$
as desired.
\end{proof}

According to the above proof and Theorem \ref{thm:LUC:Frechet}, it is easy to show the locally uniformly convergence.

\begin{corollary}\label{cor:W}
Consider $\alpha<0$. Under assumptions in Theorem \ref{thm:Weibull}, we get
$$
\sup_{x\in K} |w_n(x)-\varphi_\alpha^{\rm free}(x)| \le O(n^{-1}\lor g(\omega_F-a_n))
$$
for any compact sets $K$ in $(-1,0)$.
\end{corollary}
\begin{proof}
Put $\beta=-\alpha>0$ and use the symbols in the proof of Theorem \ref{thm:Weibull}. By Theorem \ref{thm:LUC:Frechet}, we have
\begin{align*}
\sup_{-1<x<0} x^2|w_n(x)-\varphi_\alpha^{\rm free}(x)| = \sup_{x>1} |w_n^\ast(x)-\varphi_\beta^{\rm free}(x)| \le O(n^{-1} \lor g(\omega_F- a_n)).
\end{align*}
For any compact set $K=[a,b] \subset(-1,0)$ and for any $x\in K$, we have 
\begin{align*}
|w_n(x)-\varphi_\alpha^{\rm free}(x)| \le \frac{1}{b^2} \ O(n^{-1} \lor g(\omega_F- a_n)),
\end{align*}
as desired.
\end{proof}

\begin{example}\label{ex:Weibull}
As below, since all $F$ are differentiable, we take $a_n' \in \R$ such that $F(a_n')=e^{-\frac{1}{n}}$ and put $a_n=\omega_F-a_n'$ and $b_n=\omega_F$. Let $w_n$ be the density of $x\mapsto F^{\Box\hspace{-.55em}\lor n}(a_nx+b_n)$.
\begin{enumerate}
\item[\rm (1)] ({\it Weibull distribution}) Consider $F=\Phi_\alpha$ for $\alpha< 0$. Then $\omega_F=0$. The choice of $a_n'$ leads to $a_n'=-n^{\frac{1}{\alpha}}$ for $n\in \N$. By definition of $h_\alpha$, we get $h_\alpha(x)=0$. By Theorem \ref{thm:Weibull},
$$
\sup_{-1<x<0}|w_n(x)- \varphi_\alpha^{\rm free}(x)| \le O(n^{-1}), \qquad \alpha <-1. 
$$
By a direct computation, the result can be extend to $ \alpha \le -1/2$. One can see that
$$
w_n(x)= -\alpha(-x)^{-\alpha-1} (x) = -\alpha (-x)^{-\alpha-1} \exp\left(-\frac{1}{n}(-x)^{-\alpha}\right),
$$
for $-(-n\log(1-n^{-1}))^{-1/\alpha}<x<0$. Then, for any $-1<x<0$, 
\begin{align*}
|w_n(x)- \varphi_\alpha^{\rm free}(x)| 
&= -\alpha(-x)^{-\alpha-1} \left(1- \exp\left(-\frac{1}{n}(-x)^{-\alpha} \right)\right)\\
&\le -\alpha(-x)^{-\alpha-1} \cdot \frac{1}{n}(-x)^{-\alpha}\\
&= \frac{-\alpha}{n} (-x)^{-2\alpha -1}\\
&\le \frac{-\alpha}{n},
\end{align*}
where the second inequality follows from $1-e^{-t} \le t $ for $t>0$. The desired result holds true. However, the uniform convergence does not hold for $-1/2<\alpha<0$, see Remark \ref{rem:uniformconvergesfails}.

\item[\rm (2)] Let $F$ be a distribution function such that $1-F(x)=K(\omega_F-x)^{-\alpha}$ for some $K>0$ and $\alpha<-1$. Since $F(a_n')=e^{-\frac{1}{n}}$, we have $a_n'= \omega_F- K^{\frac{1}{\alpha}} (1-e^{-\frac{1}{n}})^{-\frac{1}{\alpha}}$ (i.e. $a_n=K^{\frac{1}{\alpha}} (1-e^{-\frac{1}{n}})^{-\frac{1}{\alpha}}$). By definition of $h_\alpha$, we get
$$
h_\alpha(x)=\frac{-\alpha K (\omega_F-x)^{-\alpha}}{(1-K(\omega_F-x)^{-\alpha})(-\log (1-K(\omega_F-x)^{-\alpha}))}+\alpha, 
$$
for all $\omega_F-K^{-\frac{1}{\alpha}} < x< \omega_F$.
Then we define $g(x)=|h_\alpha(x)|$ and $g(\omega_F-a_n)=-\alpha ne^{\frac{1}{n}}(1-e^{-\frac{1}{n}}) =O(n^{-1})$. Hence,
$$
\sup_{-1<x<0} |w_n(x)-\varphi_{\alpha}^{\rm free}(x)| \le O(n^{-1})
$$
by Theorem \ref{thm:Weibull}.

\item[\rm (3)] ({\it Uniform distribution}) We define $F(x)=x \mathbf{1}_{[0,1]}(x) + \mathbf{1}_{(1,\infty)}(x)$. Then $\omega_F=1$. The choice of $a_n'$ leads to $a_n'=e^{-\frac{1}{n}}$. If we take $a_n=1-e^{-\frac{1}{n}}$ and $b_n=1$, then $F^{\Box\hspace{-.55em}\lor n}(a_n\cdot +b_n) \xrightarrow{w} \Phi_{-1}^{\rm free}$. By definition of $h_\alpha$, we get
$$
h_{-1}(x)= \frac{1-x}{x(-\log x)} -1>0, \quad 0<x<1,
$$
and define $g(x)=h_{-1}(x)$. Then $g(\omega_F-a_n)=g(e^{-\frac{1}{n}})= ne^{\frac{1}{n}}(1-e^{-\frac{1}{n}})-1 =O(n^{-1})$, and therefore 
$$
\sup_{x\in K}|w_n(x)-\varphi_{-1}^{\rm free}(x)| \le O(n^{-1})
$$
for any compact sets $K$ in $(-1,0)$ by Corollary \ref{cor:W}.

In this case, we can show that the function $w_n(x)$ converges uniformly to $\varphi_{-1}^{\rm free}(x)=\mathbf{1}_{(-1,0)}(x)$ on $(-1,0)$, with a convergence rate of order $n^{-1}$ by direct computation. Indeed, it is easy to see that
$$
w_n(x) = n(1-e^{-\frac{1}{n}}), 
$$
for $-(1-e^{-\frac{1}{n}})^{-1}<x<0$. Then
\begin{align*}
\sup_{-1<x<0} |w_n(x)- \varphi_{-1}^{\rm free}(x)| 
&= \sup_{-1<x<0}  |n(1-e^{-\frac{1}{n}}) -1| \le \frac{1}{n}.
\end{align*}
\end{enumerate}
\end{example}

\begin{remark}\label{rem:uniformconvergesfails}
We are not able to extend Theorem \ref{thm:Weibull} to $\alpha<0$.  As mentioned above, the function $w_n(x)$ in Example \ref{ex:Weibull} (1) does not converge uniformly to $\varphi_\alpha^{\rm free}(x)$ on $(-1,0)$ for $-1/2<\alpha<0$. Then, for any $-1<x<0$,
\begin{align*}
|w_n(x)- \varphi_\alpha^{\rm free}(x)| 
&= -\alpha(-x)^{-\alpha-1} \left(1- \exp\left(-\frac{1}{n}(-x)^{-\alpha} \right)\right)\\
& \ge -\alpha(-x)^{-\alpha-1}  \left\{\frac{1}{n}(-x)^{-\alpha} -\frac{1}{2n^2}(-x)^{-2\alpha} \right\}\\
&=\frac{-\alpha}{n}  (-x)^{-2\alpha-1}\left\{ 1 - \frac{1}{2n}(-x)^{-\alpha}\right\}\\
&\ge \frac{-\alpha}{n}\left( 1- \frac{1}{2n}\right)  (-x)^{-2\alpha-1}.
\end{align*}
where the second inequality follows from $1-e^{-t} \ge t -\frac{t^2}{2}$ for $t>0$. For any $n\in \N$, if we take $x$ such that $-\{-\frac{\alpha}{n} (1-\frac{1}{2n})\}^{1/(2\alpha+1)} < x< 0$, then
$$
|w_n(x)- \varphi_\alpha^{\rm free}(x)|  \ge 1.
$$

\end{remark}


\section{The case $\alpha=0$ (Gumbel type)}

In this section, we suppose a distribution function $F$ expressed by $F=\exp\{ -e^{-\phi} \}$, where $\phi$ is a twice differentiable function with $\phi'>0$ in a (left) neighborhood of $\omega_F\le \infty$. We define 
\[
f(x)=\frac{1}{\phi'(x)}.
\]
The function $f$ is called an {\it auxiliary function} of $F$.  Define $b_n=F^{\leftarrow}(e^{-\frac{1}{n}})$, leading to $F(b_n)=e^{-\frac{1}{n}}$, equivalently $\phi(b_n)=\log n$. Put  $a_n:=f(b_n) =\frac{F(b_n)}{nF'(b_n)}$. Since $\phi$ is differentiable, the function $x\mapsto F^{\Box\hspace{-.55em}\lor n}(a_nx+b_n)$ has the density $w_n$ defined by
\begin{align*}
w_n(x)= na_n \phi'(a_nx+b_n) (-\log F(a_nx+b_n)) F(a_nx+b_n),
\end{align*}
for all $A_n<x<B_n$, where
\begin{align*}
A_n:&= \frac{F^{\leftarrow}(1-n^{-1}) -b_n}{a_n} \quad \text{and} \quad A_n\to 0 \text{ as } n\to \infty,\\
B_n:&=\frac{\omega_F-b_n}{a_n}, \quad \text{and} \quad B_n\to \infty \text{ as } n\to \infty,
\end{align*}
where we understand $B_n=\infty$ if $\omega_F=\infty$.

We now assume the von Mises condition for sample distributions F.
\begin{assumption}\label{assumpt:Gumbel}  We define
\[
h_0(x):=f'(x) = -\log F(x)- \left\{ \frac{F(x)F''(x)(-\log F(x))}{(F'(x))^2}+1 \right\}.
\]
Assume $h_0(x)\to 0$ as $x\to \omega_F$ (see \cite[(2.52)]{R}).
\end{assumption}
According to \cite[Pages 114]{R}, under Assumption \ref{assumpt:Gumbel}, we have $F^n(a_n \cdot +b_n)\xrightarrow{w} \Phi_0$ as $n\rightarrow\infty$, where 
$$\Phi_0(x):=\exp(-e^{-x})\mathbf{1}_{\R}(x)$$ is the Gumbel distribution. Due to \cite[Theorem 6.11]{BV06}, one can see that $F^{\Box\hspace{-.55em}\lor n}(a_n \cdot +b_n)\xrightarrow{w}\Phi_0^{\rm free}$.

\begin{lemma}\label{lem:DG}

Under Assumption \ref{assumpt:Gumbel}, we have $F^{\Box\hspace{-.55em}\lor n}(a_n \cdot +b_n)\xrightarrow{w}\Phi_0^{\rm free}$.
\end{lemma}

Here, we focus on the following class of sample distributions.

\begin{definition}
Denote by $\calF_0$ the set of all distribution functions $F$ such that
\begin{enumerate}
\item[(G-1)] $F=\exp\{-e^{-\phi}\}$, where $\phi$ is a twice differentiable function with $\phi'>0$ in a neighborhood of $\omega_F$
\item[(G-2)] there exists a nonincreasing continuous function $g$ such that $g(x) \to 0$ as  $x\to\omega_F$ and $|h(x)|\le g(x)$ for all $x < \omega_F$.
\end{enumerate}
\end{definition}

By Lemma \ref{lem:DG}, it is easy to see that if $F\in \calF_0$, then $F^{\Box\hspace{-.55em}\lor n}(a_n \cdot +b_n)\xrightarrow{w} \Phi_0^{\rm free}$. In this case, we investigate the convergence of densities towards the free Gumbel distribution as follows.

For a positive real number $a$, we define the following distribution functions:
\begin{align*}
U_+(a,x):&=\{1-(1+ax)^{-a^{-1}}\}\mathbf{1}_{(-a^{-1},\infty)}(x),\\
U_-(a,x):&=\{1-(1-ax)^{a^{-1}}\}\mathbf{1}_{(-\infty, a^{-1})}(x) + \mathbf{1}_{[a^{-1},\infty)}(x).
\end{align*}

\begin{lemma}\label{lem:U}
For $0<a<1$, we get $$\sup_{x\in \R} |U_\pm (a,x) - \Phi_0^{\rm free}(x)| \le e^{-1}a.$$
\end{lemma}
\begin{proof}
We provide only the bound on $U_-(a,x)-\Phi_0^{\rm free}(x)$. Note that
\begin{align*}
\sup_{x\in \R}|U_- (a,x) - \Phi_0^{\rm free}(x)| &= \sup_{x\le 0}|U_--\Phi_0^{\rm free}| \lor \sup_{0<x <a^{-1}}|U_-- \Phi_0^{\rm free}| \lor  \sup_{a^{-1}\le x}|U_--\Phi_0^{\rm free}|\\
&=: I_1\lor I_2 \lor I_3.
\end{align*}
For $x\ge a^{-1}$, one can see $|U_-(a,x) - \Phi_0^{\rm free}(x)| =e^{-x} \le e^{-a^{-1}} \le e^{-1} a$, and therefore $I_3\le e^{-1}a$. Next, for $x\le 0$, it is easy to see $|U_-(a,x) - \Phi_0^{\rm free}(x)|=1-(1-ax)^{a^{-1}} \le e^{-1}a$, and hence $I_1\le e^{-1}a$. Finally, for $0< x <a^{-1}$, we have
\begin{align*}
|U_-(a,x)- \Phi_0^{\rm free}(x)|=e^{-x}-(1-ax)^{a^{-1}} =: u(x)>0.
\end{align*}
It is easy to verify that there is a unique point $0 < x_0 < a^{-1}$ such that $u'(x_0)=0$ and $u(x) \le u(x_0)$ by the intermediate theorem. Therefore $u(x)$ takes the supremum of on $(0,a^{-1})$ if and only if $u'(x_0)=0$, that is, $(1-ax_0)^{a^{-1}-1} =e^{-x_0}$. Hence
\begin{align*}
I_2&\le \sup_{0< x < a^{-1}} (e^{-x}-e^{-x}(1-ax))= a\sup_{0< x < a^{-1}} xe^{-x}\le a a^{-1} e^{-a^{-1}} \le e^{-1}a.
\end{align*}
Thus, the desired result is obtained.
\end{proof}

\begin{theorem}\label{thm:Gumbel}
Let $F\in \calF_0$, $a_n$, $b_n$, $g$, $f$ be defined as above. Then $w_n$ converges uniformly to the free Gumbel density $\varphi_0^{\rm free}$ on $(0, \infty)$. More strictly, we obtain
\[
\sup_{x>0} |w_n(x)- \varphi_0^{\rm free}(x)| \le O(n^{-1} \lor g(b_n)),
\]
for sufficiently large $n$. 
\end{theorem}

\begin{proof}
For all $x>0$ and for sufficiently large $n\in \N$ (such that infinitely many $b_n >0$), we have
\begin{align*}
|w_n(x)&-\varphi_0^{\rm free}(x)|\\
&\le |na_n\phi'(a_nx+b_n)(-\log F(a_nx+b_n))F(a_nx+b_n)-e^{-x}F(a_nx+b_n)|\\
&\hspace{4mm}+ |e^{-x}F(a_nx+b_n)-e^{-x}|\\
&\le |na_n\phi'(a_nx+b_n)(-\log F(a_nx+b_n))-e^{-x}| +(1-F(b_n))\\
&= \underbrace{\left| \frac{f(b_n)}{f(a_nx+b_n)}(-n\log F(a_nx+b_n))-e^{-x}\right|}_{=:I}+(1-e^{-\frac{1}{n}}).
\end{align*}

By \cite[(2.54)]{R}, for $x>0$, we have
\begin{align*}
U_+(g(b_n),x) \le 1+n\log F(a_nx+b_n) \le U_-(g(b_n),x).
\end{align*}
According to \cite[Lemma 2]{HR82}, for any $\epsilon>0$, there exists $n_0\in \N$ such that for $x>0$ and $n\ge n_0$, we get
$$
\frac{1}{1+\epsilon} \left( \frac{-\log F(b_n)}{-\log F(a_nx+b_n)}\right)^{-\epsilon}\le \frac{f(b_n)}{f(a_nx+b_n)} \le \frac{1}{1-\epsilon} \left( \frac{-\log F(b_n)}{-\log F(a_nx+b_n)}\right)^{\epsilon}.
$$
Hence, for any $x>0$, $n\ge n_0$,
\begin{align*}
\frac{1}{1+\epsilon} (1-U_-(g(b_n),x))^{1+\epsilon} \le  \frac{f(b_n)}{f(a_nx+b_n)} (-n \log F(a_nx+b_n)) \le \frac{1}{1-\epsilon} (1-U_+(g(b_n),x))^{1-\epsilon}.
\end{align*}
Finally, we get
\begin{align*}
\Phi_0^{\rm free}(x)-U_-(g(b_n),x)  \le \frac{f(b_n)}{f(a_nx+b_n)} (-n \log F(a_nx+b_n)) -e^{-x} \le \Phi_0^{\rm free}(x)-U_+(g(b_n),x),
\end{align*}
for sufficiently large $n$. By Lemma \ref{lem:U}, we get $I\le e^{-1}g(b_n)$, as desired.
\end{proof}

\begin{example} Since all $F$ are differentiable as below, we take $b_n$ such that $F(b_n)=e^{-\frac{1}{n}}$ and $a_n=f(b_n)$. Let $w_n$ be the density of the function $x\mapsto F^{\Box\hspace{-.55em}\lor n}(a_nx+b_n)$ as below.
\begin{enumerate}
\item ({\it Classical Gumbel distribution})
Suppose that $F=\Phi_0$. By definition of the functions $\phi$, $f$ and $h_0$, we get $\phi(x)=x$, $f(x)=1$ and $h_0(x)=f'(x)=0$. The choices of $a_n$ and $b_n$ imply that $a_n=1$ and $b_n=\log n$ for $n\in \N$, respectively. By Theorem \ref{thm:Gumbel},
\begin{align*}
\sup_{x>0} |w_n(x)-\varphi_0^{\rm free}(x)| \le O(n^{-1}). 
\end{align*}

\item Define $F(x)=\exp\{ -e^{-x^\alpha}\}$ for $\alpha>0$ with $\alpha\neq 1$. Then we get $\phi(x)=x^\alpha$, $f(x)=\alpha^{-1} x^{-\alpha+1}$ and $h_0(x)=(-1+\alpha^{-1}) x^{-\alpha}$. The choices of $a_n$ and $b_n$ lead to $a_n= \alpha^{-1} (\log n)^{-1+\frac{1}{\alpha}}$ and $b_n=(\log n)^{\frac{1}{\alpha}}$ for $n\in \N$, respectively. Define $g(x)=|h_0(x)|=|-1+\alpha^{-1}| x^{-\alpha}$. Then we get
\begin{align*}
\sup_{x>0}| w_n(x)-\varphi_0^{\rm free}(x)| \le O\left( \frac{1}{n} \lor \left|-1+\frac{1}{\alpha}\right| \frac{1}{\log n}\right) = O\left(\frac{1}{\log n}\right).
\end{align*}

\item ({\it Normal distribution})
Let $F$ be the standard normal distribution function and let us set $p(t):=\frac{1}{\sqrt{2\pi}} e^{-\frac{t^2}{2}}$ as the standard normal density.
We define $b_n$ such that $F(b_n)=e^{-\frac{1}{n}}$ and $a_n=\frac{F(b_n)}{nF'(b_n)}$. It is easy to see that $F'(x)=p(x)$ and $F''(x)=-xp(x)$, and therefore $h_0(x)=-\log F(x) ( 1+\frac{xF(x)}{p(x)}) -1$. Here we take $g(x) =1 +\log F(x) (1+ \frac{x}{p(x)}) $. Hence 
\begin{align*}
g(b_n) &= 1- \frac{1}{n} \left(1 + \frac{b_n}{p(b_n)}\right)=1- \frac{1}{n} - \frac{\sqrt{2\pi} b_n}{n} \exp\left(\frac{b_n^2}{2}\right).
\end{align*}
Since it must be $\frac{\sqrt{2\pi} b_n}{n} \exp\left(\frac{b_n^2}{2}\right) \to 1$ as $n\to \infty$, we get
$$
b_n= \sqrt{2\log n} - \frac{\log \log n+\log 4\pi}{2\sqrt{2\log n}},
$$
see e.g. \cite{Hall79}. Finally, we have $g(b_n)=O((\log n)^{-\frac{1}{2}})$, and therefore
$$
\sup_{x>0}| w_n(x)-\varphi_0^{\rm free}(x)| \le O\left( \frac{1}{n} \lor O\left(\frac{1}{\sqrt{\log n}}\right)\right) =O\left(\frac{1}{\sqrt{\log n}}\right),
$$
by Theorem \ref{thm:Gumbel}.
\end{enumerate}
\end{example}

\subsection*{Acknowledgement}

This work was started when Y.K. was affiliated with Department of Mathematics, Hokkaido University. Furthermore, Y.U. was supported by JSPS Grant-in-Aid for Young Scientists 22K13925.

\bibliographystyle{amsplain}

\end{document}